\documentclass[ECP]{ejpecp} 
  \usepackage{amsmath, latexsym, amsfonts, amssymb, amscd,bbm , stmaryrd}
\usepackage[english]{babel}
\usepackage{mathrsfs}

 \usepackage{color}

\newcommand{\norm}[1]{\left \| #1 \right \|}

 \newcommand{\lsup}[1]{\underset{#1\to\infty}{\overline{\lim}}}

\SHORTTITLE{Mean-Field Limit of Adaptive Networks}

\TITLE{The Kinetic Limit of Jump-Markov Processes on Adaptive Stochastic Networks}
 
 \AUTHORS{%
  James~MacLaurin\footnote{New Jersey Institute of Technology United States of America.
    \EMAIL{maclauri@njit.edu}}
  }

\SUBMITTED{June 1 2026} 
\ACCEPTED{December 13, 2026} 

\ABSTRACT{
We determine the large size limit of a network of interacting Mean-Field Jump-Markov Processes on an adaptive network. The flipping of the node variables is taken to have an intensity given by the mean-field of the afferent edges and nodes. The flipping of the edge variables is a function of the afferent node variables. The edge variables can be either symmetric or asymmetric. This model is motivated by applications in sociology, neuroscience and epidemiology. In general, the limiting probability law can be expressed as a fixed point of a self-consistent Poisson Process with intensity function that is (i) delayed and (ii) depends on its own probability law. In the particular case that the edge flipping is only determined by the state of the pre-synaptic neuron (as in neuroscience) it is proved that one obtains an autonomous neural-field type equation for the dual evolution of the synaptic potentiation and neural potentiation.
}
\makeatletter
\providecommand{\@doiprefix}{https://doi.org/}
\providecommand{\@DOI}{}
\makeatother
\begin{document}

We study stochastic processes on Adaptive Networks, consisting of node variables and edge variables that evolve stochastically in time \cite{Kuehn2012}. Both the node and edge variables take values in a finite state space, with Poissonian jumps between states. The intensity of the edge variable flipping is a function of the states of the nodes at either end, and the flipping of the node variables is a `mean-field' of the states of all edges and accompanying nodes. There are many applications of this sort of network \cite{Gross2009,banerjee2024co,durrett2026}: including neuroscience (the edge dynamics corresponds to slow synaptic dynamics and learning) \cite{atay2006neural,sreenivasan2019and,maclaurin2026large,coombes2023neurodynamics}, epidemiology (for example, if individuals are more likely to self-isolate if they get infected) \cite{volz2009epidemic}, social networks \cite{lanchier2024stochastic}. 

There is increasing mathematical interest in particle systems on graphs with dynamic edges in recent years \cite{durrett2026}. Volz and Meyers introduced the `neighbor exchange' SIR model for the spread of an infection, meaning that the composition of neighbours of particular agents evolves stochastically. Ball and Britton introduced a model with preventative rewiring \cite{ball2022epidemics}. Durrett introduced the adaptive voter model \cite{durrett2010random,basu2017evolving}. 

Despite the many applications, there does not seem to exist a general `Mckean-Vlasov' type equation that yields the large $n$ limiting dynamics for Hawkes Processes. Previous work, including by this author, only obtained limiting equations by resorting to implicit delay-stochastic differential equations \cite{maclaurin2026large}. There are several studies of related systems, including \cite{Gkogkas2023}, Many scholars have also considered the hydrodynamic limit of large networks of interacting neurons on inhomogeneous graphs, including \cite{Chevallier2019,Lucon2020a,AgatheNerine2022,Bramburger2023,Bramburger2024,avitabile2026neural}.
 
 Another very important application of these results is that it yields a general formalism for generating random graphs, complementing for instance \cite{Lovasz2012}. If, for example, the empirical measure concentrates at a single value in the large $n$ limit, then it will automatically yield an accurate understanding of the local structure of the graph. 

\subsection{Notation}

Throughout this paper we work in a complete probability space $(\Omega,\mathcal{F},\mathbb{P})$, equipped with a filtration $\mathbb{F} = (\mathcal{F}_t)_{t\geq 0}$ satisfying the usual conditions.
The index set of the nodes is $I_n := \lbrace 1,2,\ldots,n \rbrace$.  Let $\Gamma$ and $\Gamma_E$ be discrete sets, specifying the possible states of the node variables and edge variables. Let $\mathcal{D}([0,T], \Gamma)$ and $\mathcal{D}([0,T], \Gamma_E)$ specify the set of all cadlag trajectories taking values in (respectively) $\Gamma$ and $\Gamma_E$. This means that any $x \in \mathcal{D}([0,T], \Gamma)$ must be (i) piecewise constant, (ii) with only a finite number of discontinuities, (iii) possessing left limits and continuous from the right.

Let $\mathcal{X}_T \subset \mathcal{D}([0,T],\mathbb{R})$ consist of all cadlag functions that are (i) non-decreasing, and (ii) equal to $0$ at time $0$. For any $x,y \in \mathcal{D}([0,T], \Gamma)$, define
\begin{align}
d_t(x,y) = \int_0^t \mathbf{1}\lbrace x_s \neq y_s \rbrace ds. \label{eq: hilbert norm}
\end{align}
We note that $d_t$ generates the Skorohod Toplogy on $\mathcal{D}([0,T], \Gamma)$, however $d_t$ is not a complete metric. For a positive integer $a$, let $\mathcal{Z}_{T,a}$ consist of all $x \in \mathcal{D}([0,T], \Gamma)$ such that the number of points of discontinuity is less than or equal to $a$. One can check that $\mathcal{Z}_{T,a}$ is a compact subset of $ \mathcal{D}([0,T], \Gamma)$, and that $d_t$ is complete over  $\mathcal{Z}_{T,a}$. 

Throughout this paper, $\mathcal{E}$ is a compact Riemannian Manifold, and $d_{\mathcal{E}}: \mathcal{E} \times \mathcal{E} \to \mathbb{R}_{\geq 0}$ is a complete metric on $\mathcal{E}$.  Let $B_{\epsilon}(\theta) \subset \mathcal{E}$ be the open ball about $\theta \in \mathcal{E}$ of radius $\epsilon$.
Let $d_{W,T}: \mathcal{P}\big( \mathcal{E} \times  \mathcal{D}([0,T], \Gamma) \big) \times \mathcal{P}\big( \mathcal{E} \times  \mathcal{D}([0,T],\Gamma) \big)  \mapsto \mathbb{R}$ be the Wasserstein Distance, i.e 
\begin{equation}
d_W(\mu,\nu) = \inf_{\zeta} \big\lbrace \mathbb{E}^{\zeta}\big[ d_{\mathcal{E}}(\theta,\tilde{\theta}) + d_T(x,y) \big] \big\rbrace,
\end{equation}
where the infimum is over all couplings $\zeta \in \mathcal{P}\big( \mathcal{E} \times  \mathcal{D}([0,T], \Gamma) \times  \mathcal{E} \times  \mathcal{D}([0,T], \Gamma) \big)$. 

We let $\mathcal{Y} \subseteq [0,1]^{\Gamma_E \times \Gamma}$ consist of all $(g_{a,\alpha})_{a\in \Gamma_E , \alpha \in \Gamma}$ such that
 \begin{align}
 \sum_{a\in \Gamma_E} \sum_{\alpha\in\Gamma} g_{a,\alpha} = 1.
 \end{align}
Throughout this paper we will identify $\mathcal{Y}$ with $\mathcal{P}(\Gamma_E \times \Gamma)$.



\section{Model Outline}
We consider a network of inhomogeneous Poisson Processes on a random adaptive network. Node $j \in I_n := \lbrace 1,\ldots, n \rbrace$ is assigned a position $\theta^j_n$ in a compact Riemannian Manifold $\mathcal{E}$.

\begin{hypothesis}
There exists a measure $\mu_{\mathcal{E}} \in \mathcal{P}(\mathcal{E})$ such that
\begin{align}
\lim_{n\to\infty} n^{-1}\sum_{j\in I_n} \delta_{\theta^j_n} = \mu_{\mathcal{E}}.
\end{align}
It is also assumed that $\mu_{\mathcal{E}}$ is absolutely continuous with respect to the metric measure $\mu_{\mathcal{E}}$ on $\mathcal{E}$, with continuous density.
\end{hypothesis}

It is assumed that the edges $\lbrace J_t^{ jk} \rbrace_{j , k \in I_n \fatsemi 1\leq  q \leq p}$ take values in a discrete set $\Gamma_E$, and that the node variables $\lbrace \sigma^j_t \rbrace_{j\in I_n}$ take values in a discrete set $\Gamma$.  We assume that the edge-transitions depend on the states of the end variables, i.e. for $h\ll 1$ and $\alpha \in \Gamma_E$, if $ J_{t}^{j  k} = b$ then
 \begin{align}
 \mathbb{P}\big(  J_{t+h}^{j k} = a \; \big| \; \mathcal{F}_t \big) = h l_{ b \mapsto a} \big( \sigma^j_t ,  \sigma^{k}_t \big) + O(h^2) .
 \end{align}
 Define $\hat{\mu}^{n,j}_t \in \mathcal{Y} := \mathcal{P}(\Gamma_E \times \Gamma)$ to represent the local empirical measure containing information about the neighborhood of node $j$. It is such that, writing, for $\zeta \in \Gamma$ and $a\in \Gamma_E$,
\begin{align}
\hat{\mu}^{n,j}_t(a,\zeta) =& n^{-1}\sum_{k\in I_n} \mathbf{1} \big\lbrace J_t^{j k} =a \big\rbrace \mathbf{1}\big\lbrace  \sigma^{k}_t = \zeta \big\rbrace .
\end{align}
It is assumed that the transitions of the nodes are Poissonian, and such that for all $\alpha \neq \beta$, there exists a Lipschitz function
 \begin{align}
 f_{\alpha \mapsto \beta}: \mathcal{Y} \mapsto  \mathbb{R}^+ \label{eq: f transition function}
 \end{align}
 such that if $\sigma_t^j = \alpha$ then for $h \ll 1$,
 \begin{align}
  \mathbb{P}\big(  \sigma_{t+h}^{j} = \beta \; \big| \; \mathcal{F}_t \big) =h f_{\alpha \mapsto \beta}\big(  \hat{\mu}^{n,j}_t  \big) + O(h^2).
 \end{align}
This sort of model is well-motivated in neuroscience and epidemiology. See for instance \cite{maclaurin2026large2} for further details.
\begin{hypothesis} \label{Edge Jump Assumptions}
We assume that either (\textbf{A}) with unit probability, for all $t \geq 0$, $J^{jk}_t = J^{kj}_t$, or that (\textbf{B}) the transitions are independent, i.e. for $j \neq k$ and any $T > 0$,
\begin{align}
\lim_{h\to 0} h^{-1} \mathbb{P}\big( \text{ For some }t\leq T, \;   J_{t+h}^{j k} \neq J^{jk}_t \text{ and }J_{t+h}^{kj} \neq J^{kj}_t    \big) = 0.
\end{align}
\end{hypothesis}
 The initial conditions $\lbrace \sigma_0^j \rbrace_{j\in I_n}$ and $\lbrace J^{jk}_0 \rbrace_{j,k\in I_n}$ are constants (and they can also depend on $n$, although its neglected from the notation), and it is assumed that the empirical distribution of the initial conditions and spatial locations converges weakly, i.e. it is assumed that
 \begin{hypothesis}
There exists a continuous function $v: \mathcal{E} \times \mathcal{E} \mapsto \mathcal{P}(\Gamma \times \Gamma \times \Gamma_E)$ such that for any continuous function $\mathcal{B}: \mathcal{E}\times\mathcal{E} \times \Gamma \times \Gamma \times \Gamma_E \mapsto \mathbb{R}$,
 \begin{multline}
\sup_{\alpha_1,\alpha_2\in\Gamma} \sup_{a\in \Gamma_E} \lim_{n\to\infty} \bigg| \int_{\mathcal{E}}  \int_{\mathcal{E}} \mathcal{B}(x,y,\alpha_1,\alpha_2,a) v_{xy}(\alpha_1,\alpha_2,a) d\mu_{\mathcal{E}}(x) d\mu_{\mathcal{E}}(y) \\ - n^{-2}\sum_{j,k\in I_n} \mathcal{B}(\theta^j_n,\theta^k_n,\alpha_1,\alpha_2,a) \mathbf{1}\lbrace \sigma^j_0 = \alpha_1, \sigma^k_0 = \alpha_2 , J^{jk}_0 = a  \rbrace \bigg| = 0.
 \end{multline}
 \end{hypothesis}

 \subsection{ Main Results}
 We first state our main result. 
 \begin{theorem} \label{Main Theorem}
 There exists $\mu \in \mathcal{P}\big( \mathcal{E} \times \mathcal{D}([0,T],\Gamma) \big)$ (this is precisely specified in Theorem \ref{Theorem Existence of Limit Law} below), such that for any $\epsilon > 0$,
 \begin{align}
 \lsup{n} n^{-1}\log \mathbb{P}\big( d_W(\hat{\mu}^n , \mu) \geq \epsilon \big) < 0.
 \end{align}
 \end{theorem}
 The limiting probability law $\mu$ is (in general) non-autonomous. This makes it difficult to analyze and identify phase transitions and bifurcations (most of the dynamical systems machinery is geared towards autonomous systems). However if we make the additional assumption that the transition rate of the edge $J^{jk}_t$ is independent of $\sigma^j_t$ (which in general implies that $J^{jk}_t \neq J^{kj}_t$), then we obtain an autonomous expression. This greatly facilitates an analysis of pattern formation and phase transitions.
 \begin{hypothesis} \label{Hypothesis Additional Autonomous}
 Suppose that (\textbf{A}) of Hypothesis \ref{Edge Jump Assumptions} holds, and that in addition, the intensity of the connectivity changes only depends on one of the edges, i.e.
 \begin{align}
  l_{ b \mapsto a} \big( \sigma^j_t ,  \sigma^{k}_t \big) :=   \tilde{l}_{ b \mapsto a} \big(  \sigma^{k}_t \big).
\end{align}
\end{hypothesis}
Assumption \ref{Hypothesis Additional Autonomous} is well-motivated in neuroscience, since there is a directionality to the signal propagation.
  \begin{theorem} \label{Theorem Convergence Empirical Measure Autonomous}
  Suppose that Hypothesis \ref{Hypothesis Additional Autonomous} holds in addition to the other hypotheses. Then for any continuous function $\mathcal{H}: \mathcal{E}\times\mathcal{E}\times \Gamma\times\Gamma_E \mapsto \mathbb{R}$ and any $T > 0$, $\mathbb{P}$-almost-surely,
  \begin{multline}
  \lim_{n\to\infty} \sup_{\alpha \in \Gamma}\sup_{a\in\Gamma_E} \sup_{t\leq T}\bigg| \int_{\mathcal{E}} \int_{\mathcal{E}} \mathcal{H}(x,y,\alpha,a) p_{xy}(\alpha,a)d\mu_{\mathcal{E}}(x) d\mu_{\mathcal{E}}(y) \\- n^{-2}\sum_{j,k\in I_n} \mathcal{H}(\theta^j_n,\theta^k_n , \alpha,a) \mathbf{1}\lbrace \sigma^k_t = \alpha , J^{jk}_t = a \rbrace \bigg| =0.
  \end{multline}
Here $p_{\theta\eta,t}$ is the unique solution of the following system of PDEs. For $\theta,\eta \in \mathcal{E}$, and $\alpha \in \Gamma$, $a\in \Gamma_E$,  
  \begin{align} \label{eq: p theta eta t initial}
  p_{\theta\eta,0}(\alpha,a) =  \sum_{\beta\in\Gamma} v_{\theta\eta}(\beta,\alpha,a) 
  \end{align}
  and for $t> 0$,
  \begin{multline} \label{eq: p theta eta t evolution}
  \frac{d}{dt}p_{\theta\eta,t}(\alpha,a) = \sum_{\zeta\in\Gamma : \zeta \neq \beta}\big( f_{\zeta \mapsto \alpha}(G_{\eta,t})p_{\theta\eta,t}(\zeta,a) -  f_{\alpha \mapsto \zeta}(G_{\eta,t})p_{\theta\eta,t}(\alpha,a)  \big) \\
  + \sum_{b\in \Gamma_E \; : b\neq a} \big( \tilde{l}_{b\mapsto a}(\alpha) p_{\theta\eta,t}(\alpha,b)  -  \tilde{l}_{a\mapsto b}(\alpha) p_{\theta\eta,t}(\alpha,a)  \big)
  \end{multline}
  where $G_{\theta,t} \in \mathcal{P}(\Gamma\times\Gamma_E)$ is such that for $\alpha\in\Gamma$, $a\in \Gamma_E$,
  \begin{align}
  G_{\theta,t}(\alpha,a) =  \int_{\mathcal{E}}  p_{\theta\eta,t}(\alpha,a) d\mu_{\mathcal{E}}(\eta)
  \end{align}
 \end{theorem}
  
 \begin{remark}
 Theorem \ref{Theorem Convergence Empirical Measure Autonomous} is equivalent to
 \[
 \lim_{n\to\infty} \sup_{t\leq T} \big\lbrace d_W\big(\tilde{\mu}^n_t , \nu_t) \big\rbrace= 0
 \]
 where
 \[
 \tilde{\mu}^n_t = n^{-2}\sum_{j,k\in I_n} \delta_{\theta^j_n , \theta^k_n, J^{jk}_t , \sigma^k_t} \in \mathcal{P}\big( \mathcal{E} \times \mathcal{E} \times \Gamma_E \times \Gamma \big),
 \]
 and $\nu_t \in \mathcal{P}\big( \mathcal{E} \times \mathcal{E} \times \Gamma_E \times \Gamma\big)$ is such that for measurable $A,B \subseteq \mathcal{E}$, and any $a\in\Gamma_E$ and $\alpha\in\Gamma$,
 \[
 \nu_t(A\times B \times a \times \alpha) = \int_{A}\int_{B} p_{\theta\eta,t}(\alpha,a) d\mu_{\mathcal{E}}(\theta) d\mu_{\mathcal{E}}(\eta) .
 \]
 \end{remark}
 We outline the proof of Theorem \ref{Theorem Convergence Empirical Measure Autonomous}.
 \begin{proof}
A consequence of Hypothesis \ref{Hypothesis Additional Autonomous} is that the function $\psi_{\theta,t}: \mathcal{D}([0,t],\Gamma) \times \mathcal{P}\big( \mathcal{D}([0,t],\Gamma) \big) $  ( defined in \eqref{eq: psi definition}) is independent of its first argument, and so we can write
\[
\tilde{\psi}_{\theta,t}(\mu) := \psi_{\theta,t}(\cdot , \mu).
\]
Lemma \ref{Lemma convergence of tilde mu n t} implies that for any $\epsilon > 0$,
\begin{align}
\lsup{n} n^{-1}\log \mathbb{P}\big(  n^{-1} \sum_{j\in I_n} \sup_{t\leq T} d_W\big(\hat{\mu}^{n,j}_t , \tilde{\psi}_{\theta^j_n,t}(\mu) \big) \geq \epsilon \big) < 0,
\end{align}
and therefore
\[
\lim_{n\to\infty} n^{-1} \sum_{j\in I_n} \sup_{t\leq T} d_W\big(\hat{\mu}^{n,j}_t , \tilde{\psi}_{\theta^j_n,t}(\mu) \big) = 0.
\]
\end{proof}

 \section{The Kinetic Limit}
We must first specify the limiting probability law for Theorem \ref{Main Theorem}. It satisfies a nonlinear implicit equation with delays. We start by specifying how the average connectivity at time $t$ is determined by (i) the initial value of the connectivity and (ii) the trajectories of the afferent spin variables upto time $t$.

\begin{lemma}
For any $c \in \Gamma_E$ any $\theta \in \mathcal{E}$, any $z , y \in \mathcal{D}([0,T] , \Gamma)$, there exists a unique $ \lbrace j_{\theta,\eta} \rbrace_{\theta,\eta \in \mathcal{E}}  \in   \mathcal{D}([0,T] , \mathcal{P}( \Gamma_E) ) $, written $ j_{\theta\eta} := (j_{\theta\eta,t})_{t\leq T}$ such that for all $t\leq T$, and all $a \in \Gamma_E$,
\begin{align}
j_{\theta\eta,t}(a) = v_{\theta\eta}(z_0, y_0,a)+ \sum_{b \in \Gamma_E : b \neq a} \int_0^t \bigg\lbrace j_{\theta\eta,s}(b) l_{b\mapsto a}(z_s , y_s)   - j_{\theta\eta,s}(a)    l_{a\mapsto b}(z_s , y_s)  \bigg\rbrace ds
\end{align}
We write
\begin{align}
\Psi_{\theta \eta , t} &:  \mathcal{D}([0,t] , \Gamma)^2 \mapsto \mathcal{D}\big([0,t] , \mathcal{P}( \Gamma_E) \big) \\
\Psi_{\theta \eta,t}(  z , y) &:=   j_{\theta\eta}.
\end{align}
Furthermore for each $T > 0$, there exists a constant $C_T$ such that for all $t\leq T$,
\begin{align}
\sup_{\theta,\eta\in \mathcal{E}} \norm{ \Psi_{\theta \eta,t}(  z , x) -  \Psi_{\theta \eta,t}(  \tilde{z} , \tilde{x}) } \leq C_T \big( \norm{z - \tilde{z}}_t + \norm{x - \tilde{x}}_t \big). \label{Lipschitz Bound}
\end{align}
\end{lemma}
\begin{proof}
By definition of $\mathcal{D}([0,T], \Gamma)$, both $y,z$ must be piecewise constant, with only a finite number of discontinuities. One then finds a unique solution to the ODE for $j_{\theta,\eta,t}$ along any interval over which both $y$ and $z$ are constant, and iteratively one finds a unique solution upto time $T$.

To see why \eqref{Lipschitz Bound} is true, write $\tilde{j}_{\theta\eta,t}(a) = \Psi_{\theta \eta,t}(  \tilde{z} ,\tilde{x} )$. Then one immediately finds that there is a universal constant $c$ such that
\begin{align}
\sup_{a\in \Gamma}\big| j_{\theta\eta,t}(a) - \tilde{j}_{\theta\eta,t}(a) \big| \leq c \int_0^t \sup_{a\in \Gamma}\big| j_{\theta\eta,s}(a) - \tilde{j}_{\theta\eta,s}(a) \big| ds + c \norm{z - \tilde{z}}_t + c \norm{x - \tilde{x}}_t
\end{align}
Gronwall's Inequality now implies
\begin{align}
\sup_{a\in \Gamma}\big| j_{\theta\eta,t}(a) - \tilde{j}_{\theta\eta,t}(a) \big| \leq c \exp(ct) \big(  \norm{z - \tilde{z}}_t + c \norm{x - \tilde{x}}_t \big).
\end{align}
\end{proof}
For $\theta \in \mathcal{E}$, we next define
\begin{align} \label{eq: psi definition}
\psi_{\theta,t}:  \mathcal{D}([0,t] , \Gamma) \times \mathcal{P}\big( \mathcal{E} \times \mathcal{D}([0,t] , \Gamma) \big) \mapsto \mathcal{Y}
\end{align}
to be such that for $a \in \Gamma_E$ and $\sigma \in \Gamma$,
\begin{align}
\psi_{\theta,t}(z,\mu)(a,\sigma) =   \mathbb{E}^{(\eta,y) \sim \mu}\big[  \Psi_{\theta\eta,t}(z , y)(a) \mathbf{1}\lbrace y_t = \sigma \rbrace \big]  .
\end{align}
To check that $\psi_{\theta,t}(z,\mu)$ is a well-defined probability measure, notice that
\begin{align}
\sum_{a \in \Gamma_E} \sum_{\sigma \in \Gamma}\psi_{\theta,t}(z,\mu)(a,\sigma) =&  \sum_{a \in \Gamma_E} \sum_{\sigma \in \Gamma} \mathbb{E}^{(\eta,y) \sim \mu}\big[  \Psi_{\theta\eta,t}(z , y)(a) \mathbf{1}\lbrace y_t = \sigma \rbrace \big]   \nonumber \\
=&\sum_{\sigma \in \Gamma} \mathbb{E}^{(\eta,y) \sim \mu}\big[ \mathbf{1}\lbrace y_t = \sigma \rbrace \big] \nonumber \\
=& 1
\end{align}
and in the second line we have used the fact that
\[
\sum_{a\in \Gamma_E} \Psi_{\theta\eta,t}(z , y)(a) = 1.
\]
 In the following Theorem we outline an implicit definition of the limiting probability law $\mu_T \in  \mathcal{P}\big(\mathcal{E} \times \mathcal{D}([0,T], \Gamma) \big)$.

\begin{theorem} \label{Theorem Existence of Limit Law}
Let $\lbrace y_{\theta, \alpha\mapsto\beta}(t) \rbrace_{\theta\in \mathcal{E} \fatsemi \alpha,\beta \in \Gamma}$ be independent unit intensity Poisson counting processes.  For any $T > 0$, there exists a unique set of $\Gamma$-valued stochastic processes $\lbrace z_{\theta}(t) \rbrace_{\theta\in \mathcal{E}, t\leq T}$ that satisfy the following properties. Let $\lbrace z_{\theta,0} \rbrace_{\theta\in \mathcal{E}}$ be independent $\Gamma$-valued variables, such that
\begin{align}
\mathbb{P}\big( z_{\theta}(0) = \alpha \big) = \sum_{\beta\in\Gamma} \sum_{a\in \Gamma_E} \int_{\mathcal{E}} v_{\theta\eta}(\alpha,\beta,a) d\mu_{\mathcal{E}}(\eta).
\end{align}
For $t > 0$, define $z_{\theta}(t) = \alpha \in \Gamma$ precisely when
\begin{multline}
\mathbf{1}\lbrace z_{\theta}(0) = \alpha \rbrace + \sum_{\beta \neq \alpha} \bigg\lbrace y_{\theta, \beta\mapsto\alpha}\bigg( \int_0^t \mathbf{1}\lbrace z_{\theta}(s) = \beta \rbrace f_{\beta\mapsto\alpha}\big( \psi_{\theta,s}(z_{\theta},\mu_s) \big) ds \bigg)\\
- y_{\theta, \alpha\mapsto\beta}\bigg( \int_0^t \mathbf{1}\lbrace z_{\theta}(s) = \alpha \rbrace f_{\alpha\mapsto\beta}\big( \psi_{\theta,s}(z_{\theta},\mu_s) \big) ds \bigg) \bigg\rbrace = 1,
\end{multline}
and in the above $\mu_s \in \mathcal{P}\big(\mathcal{E} \times \mathcal{D}([0,s], \Gamma) \big)$ is such that for any measurable $A \subset \mathcal{E}$ and any measurable $B \subset  \mathcal{D}([0,s], \Gamma)$,
\[
\mu_s(A \times B) = \int_A \mathbb{P}\big( z_{\theta}([0,s]) \in B \big) d\mu_{\mathcal{E}}(\theta).
\]

\end{theorem}
\begin{proof}
Let $\mathcal{A}_t \subset \mathcal{P}\big( \mathcal{E} \times \mathcal{D}([0,t],\Gamma) \big)$ consist of all probability measures $\mu$ such that, writing $\mu$ to be the law of $(\theta,z)$, (i) it holds that for any measurable set $B \subset \mathcal{E}$,
\[
\mu(\theta \in B) = \mu_{\mathcal{E}}(B).
\]
$\mu$ must also be such that (ii)
\begin{align}
\mu\big( \theta \in B , z_0 =\alpha \big) = \sum_{\beta\in\Gamma , a\in\Gamma_E} \int_{\mathcal{E}} v_{\theta\eta}(\alpha,\beta,a) d\mu_{\mathcal{E}}(\eta).
\end{align}
Finally it is required that (iii) we have the following uniform bound on the expected number of transitions: for all $\theta\in\mathcal{E}$,
\begin{align}
\mathbb{E}^{\mu}\big[ \big| \big\lbrace s\leq t : z(s^-) \neq z(s) \big\rbrace \big| \; \big| \; \theta  \big] \leq t \sup_{\alpha,\beta \in \Gamma} \sup_{y \in \mathcal{Y}}f_{\alpha\mapsto\beta}(y)
\end{align}
where we recall that $f_{\alpha\mapsto\beta}$ is defined in \eqref{eq: f transition function} and gives the transition intensities.

For some $t > 0$, let $\mu,\tilde{\mu} \in \mathcal{A}_t$ be any two probability measures. We are going to define a coupling of these two measures via inhomogeneous counting processes $\lbrace z_{\theta}(t) , \tilde{z}_{\theta}(t) \rbrace_{\theta\in \mathcal{E}}$. Write
\begin{align}
g_{\theta,s}=& \psi_{\theta,s}(z_{\theta}, \mu) \\
 \tilde{g}_{\theta,s} =&\psi_{\theta,s}( \tilde{z}_{\theta}, \tilde{\mu}) .
\end{align}
Now define
\begin{align}
u_{\beta \mapsto \alpha}(t , \theta) &=  f_{\beta \mapsto \alpha}(g_{\theta,t}) \mathbf{1}\lbrace z_{\theta}(t) = \beta \rbrace \\
\tilde{u}_{\beta \mapsto \alpha}(t , \theta) &=  f_{\beta \mapsto \alpha}(\tilde{g}_{\theta,t}) \mathbf{1}\lbrace \tilde{z}_{\theta}(t) = \beta \rbrace \\
\bar{u}_{\beta\mapsto \alpha}(t,\theta) &= \inf\big\lbrace u_{\beta \mapsto \alpha}(t , \theta)  , \tilde{u}_{\beta \mapsto \alpha}(t , \theta)  \big\rbrace \\
\hat{u}_{\beta\mapsto\alpha}(t,\theta) &= u_{\beta \mapsto \alpha}(t , \theta) - \bar{u}_{\beta\mapsto \alpha}(t,\theta) \\
\breve{u}_{\beta\mapsto\alpha}(t,\theta) &= u_{\beta \mapsto \alpha}(t , \theta) - \bar{u}_{\beta\mapsto \alpha}(t,\theta) .
\end{align}
For independent unit-intensity Poisson counting processes $\big\lbrace \bar{Y}_{\beta\mapsto\alpha , \theta}(t) ,  \hat{Y}_{\beta\mapsto\alpha , \theta}(t) ,  \breve{Y}_{\beta\mapsto\alpha, \theta}(t) \big\rbrace_{\theta\in\mathcal{E}}$, we define
\begin{align}
\bar{Z}_{\beta\mapsto\alpha , \theta}(t) &= \bar{Y}_{\beta\mapsto\alpha , \theta}\bigg( \int_0^t  \bar{u}_{\beta\mapsto \alpha}(s,\theta) ds \bigg) \\
\hat{Z}_{\beta\mapsto\alpha , \theta}(t) &= \hat{Y}_{\beta\mapsto\alpha , \theta}\bigg( \int_0^t  \hat{u}_{\beta\mapsto \alpha}(s,\theta) ds \bigg) \\
\breve{Z}_{\beta\mapsto\alpha , \theta}(t) &= \breve{Y}_{\beta\mapsto\alpha , \theta}\bigg( \int_0^t  \breve{u}_{\beta\mapsto \alpha}(s,\theta) ds \bigg) \\
Z_{\beta\mapsto\alpha,\theta}(t) &= \bar{Z}_{\beta\mapsto\alpha , \theta}(t) + \hat{Z}_{\beta\mapsto\alpha , \theta}(t) \\
\tilde{Z}_{\beta\mapsto\alpha,\theta}(t) &= \bar{Z}_{\beta\mapsto\alpha , \theta}(t) + \breve{Z}_{\beta\mapsto\alpha , \theta}(t) .
\end{align}
Finally it is stipulated that $z_{\theta}(t) = \beta$ if and only if 
\begin{align}
\mathbf{1}\lbrace z_{\theta}(0) = \beta \rbrace + \sum_{\alpha\neq \beta} \big( \bar{Z}_{\alpha\mapsto\beta , \theta}(t) - \bar{Z}_{\beta\mapsto\alpha , \theta}(t) \big) = 1.
\end{align}
This is well-defined because the LHS of the above equals the number of transitions to $\beta$ upto time $t$, minus the number of transitions away, plus one if the initial value is equal to $\beta$. We similarly stipulate that $\tilde{z}_{\theta}(t) = \beta$ if and only if 
\begin{align}
\mathbf{1}\lbrace z_{\theta}(0) = \beta \rbrace + \sum_{\alpha\neq \beta} \big( \bar{Z}_{\alpha\mapsto\beta , \theta}(t) - \bar{Z}_{\beta\mapsto\alpha , \theta}(t) \big) = 1.
\end{align}
Write $\mu^{(1)}_t , \tilde{\mu}^{(1)}_t \in \mathcal{P}\big(\mathcal{E} \times \mathcal{D}([0,t],\Gamma) \big)$ to be the respective probability laws. That is, for any measurable $A \subseteq \mathcal{E}$, and any measurable $B \subseteq \mathcal{D}([0,t],\Gamma)$,
\begin{align*}
\mu^{(1)}_t(A \times B) =& \int_{A} \mathbb{P}\big( z_{\theta} \in B \big) d\mu_{\mathcal{E}}(\theta) \\
\tilde{\mu}^{(1)}_t(A \times B) =& \int_{A} \mathbb{P}\big( \tilde{z}_{\theta} \in B\big) d\mu_{\mathcal{E}}(\theta) .
\end{align*}
We next claim that there exists a constant $C > 0$ such that
\begin{align}
 \int_{\mathcal{E}} \mathbb{E}\big[ d_t( z_{\theta} , \tilde{z}_{\theta}) \big] \leq Ct d_W(\mu_t, \tilde{\mu}_t)  . \label{eq: to prove fixed point identity }
\end{align}
With an aim of proving \eqref{eq: to prove fixed point identity }, it follows from the definitions that there exists a constant $c > 0$ such that for all $t\leq T$,
\begin{align}
\int_{\mathcal{E}} \mathbb{E}\big[ \norm{ g_{\theta,t} - \tilde{g}_{\theta,t} } \big] d\theta \leq c\int_{\mathcal{E}} \mathbb{E}\big[ \norm{ z_{\theta} - \tilde{z}_{\theta} }_t \big] d\theta + c d_W(\mu_t, \tilde{\mu}_t).
\end{align}
Substituting the definitions, and employing Gronwall's Inequality, we then find that there is a constant $C > 0$ such that for all $t \leq T$,
\begin{align}
\sup_{\alpha\neq \beta} \int_{\mathcal{E}} \mathbb{E}\big[\hat{Z}_{\beta\mapsto\alpha , \theta}(t) + \breve{Z}_{\beta\mapsto\alpha , \theta}(t) \big] d\theta & \leq C\int_0^t \int_{\mathcal{E}} \mathbb{E}\big[ \mathbf{1}\lbrace z_{\theta}(s) \neq \tilde{z}_{\theta}(s) \rbrace  \big] d\theta ds + C d_W(\mu_t, \tilde{\mu}_t) \nonumber \\
=&C \int_{\mathcal{E}} \mathbb{E}\big[  d_t\big(  z_{\theta}, \tilde{z}_{\theta} \big) \big] d\theta  + C d_W(\mu_t, \tilde{\mu}_t) 
\end{align}
Notice finally that
\begin{align}
 \int_{\mathcal{E}} \mathbb{E}\big[  d_t\big( z_{\theta} , \tilde{z}_{\theta} \big) \big] d\theta 
\leq |\Gamma|^2 \sup_{\alpha\neq \beta} \int_{\mathcal{E}} \mathbb{E}\big[\hat{Z}_{\beta\mapsto\alpha , \theta}(t) + \breve{Z}_{\beta\mapsto\alpha , \theta}(t) \big] d\theta 
\end{align}
We can thus conclude that \eqref{eq: to prove fixed point identity } holds.  \eqref{eq: to prove fixed point identity } implies that
\begin{align}
d_{W}\big( \mu^{(1)}_t , \tilde{\mu}^{(1)}_t \big) \leq tC d_W(\mu_t, \tilde{\mu}_t).
\end{align}
The fixed point theorem then implies that for small enough $t$, there is a unique $\mu_t$ such that $\mu^{(1)}_t = \mu_t$.

We can then iterate this method for larger and larger $t$, obtaining a unique fixed point upto time $T$.
\end{proof}

\subsection{$n$-dimensional Approximation}

We now define an $n$-dimensional system of Hawkes Processes $\lbrace \tilde{\sigma}^j(t) \rbrace_{j\in I_n}$ that approximates the original system, except that the intensity of the flipping of the node-variables  $ \tilde{\sigma}^j(t) $ is independent of $\tilde{\sigma}^k$ if $k \neq j$, with the same intensity function as the limiting law $\mu$ defined in Theorem \ref{Theorem Existence of Limit Law}. This intermediate approximation will serve as a bridge between the original $n$-dimensional system and the final limiting law. To this end, let $\big\lbrace \tilde{\sigma}^j_t \big\rbrace_{j\in I_n}$ be independent jump-Markov Processes such that (i) $\tilde{\sigma}^j_0 = \sigma^j_0$ , (ii) $\tilde{J}^{jk}_0 = J^{jk}_0$, and (iii) for $b\neq a$, if $\tilde{\sigma}_t^j = a$ then 
 \begin{align}
  \mathbb{P}\big(  \tilde{\sigma}_{t+h}^{j} = \beta \; \big| \; \mathcal{F}_t \big) =h f_{\alpha \mapsto \beta}(  \hat{G}^j_{t} ) + O(h^2)
 \end{align}
 where
 \begin{equation}
\hat{G}^j_{t} = \psi_{\theta^j_n,t}( \tilde{\sigma}^j ,\mu_t)
 \end{equation}
and $\mu_t \in \mathcal{P}\big(\mathcal{E} \times \mathcal{D}([0,t], \Gamma) \big)$ is defined in Theorem \ref{Theorem Existence of Limit Law}.

 The transitions of the connectivities $\lbrace \tilde{J}^{jk}_t \rbrace_{j,k \in I_n}$ are such that for $h\ll 1$ and $\alpha \in \Gamma_E$, if $\tilde{J}_{t}^{j  k} = b$ then
 \begin{align}
 \mathbb{P}\big(  \tilde{J}_{t+h}^{j k} = a \; \big| \; \mathcal{F}_t \big) = h l_{ b \mapsto a} \big( \tilde{\sigma}^j_t ,  \tilde{\sigma}^{k}_t \big) + O(h^2) .
 \end{align}
 Define the associated empirical measure
 \begin{align}
 \tilde{\mu}^n_t = n^{-1} \sum_{j\in I_n} \delta_{\theta^j_n ,  \tilde{\sigma}^j_{[0,t]}} \in \mathcal{P}\big( \mathcal{E} \times \mathcal{D}([0,t] , \Gamma) \big).
 \end{align}
Define also, $t\leq T$, $\alpha \in \Gamma$ and $a \in \Gamma_E$, 
\begin{align}
\tilde{G}^j_t(\alpha,a) =&  n^{-1}  \sum_{i \in I_n} \mathbf{1}\big\lbrace \tilde{\sigma}^i_t = \alpha , \tilde{J}^{ji}_t = a \big\rbrace 
\end{align}
 Lets first notice that the empirical measure generated by this system converges to the same limiting law.
 \begin{lemma} \label{Lemma convergence of tilde mu n t}
 For any $\epsilon > 0$, define the event
 \begin{align}
\mathcal{U}^n_{\epsilon} = \bigg\lbrace  d_W\big( \tilde{\mu}^n_{T} , \mu_T \big) \leq \epsilon \bigg\rbrace  .
 \end{align}
 Then for any $\epsilon > 0$,
 \begin{align}
 \lsup{n} n^{-1}\log \mathbb{P}\big(  (\mathcal{U}^n_{\epsilon})^c \big) < 0. \label{eq: to show u n c epsilon}
 \end{align}
Furthermore, for any $\epsilon > 0$, 
 \begin{align}
 \lsup{n} n^{-1} \log \mathbb{P}\bigg(  \sup_{j\in I_n} \sup_{t\leq T} \sup_{a\in \Gamma_E , \alpha\in\Gamma} \big| \tilde{G}^{j}_{t}(\alpha,a)- \hat{G}^j_{t}(\alpha,a) \big| > \epsilon \bigg) < 0. \label{eq: rate of convergence of G}
 \end{align}
 \end{lemma}
 \begin{proof}
 The variables $\lbrace \tilde{\sigma}^j \rbrace_{j\in I_n}$ are independent, and the probability law of $\tilde{\sigma}^j$ depends continuously on $\theta^j_n$. It is known that the law of $\tilde{\mu}^n_{t} $ satisfies a Large Deviation Principle \cite{Comets1989,dembo2009large}, with rate function $\mathcal{I}: \mathcal{P}\big( \mathcal{D} \times  \mathcal{D}( [0,T] , \Gamma ) \big) \mapsto \mathbb{R}$, defined as follows. Suppose first that $\nu$ is such that for any measurable $A \subset \mathcal{E}$ and any measurable $B \subset  \mathcal{D}([0,T], \Gamma)$, there exists a measurable map $\theta \mapsto \nu_{\theta} \in  \mathcal{P}\big(   \mathcal{D}( [0,T] , \Gamma ) \big) $ such that 
\[
\nu(A \times B) = \int_A  \nu_{\theta}(B) d\mu_{\mathcal{E}}(\theta).
\]
In this case, the rate function is such that
 \begin{align}
 \mathcal{I}(\nu) = \int_{\mathcal{E}} \mathcal{R}( \nu_\theta || \mu_{\theta} ) d\mu_{\mathcal{E}}(\theta).
 \end{align}
 If $\nu$ does not admit this decomposition, then $\mathcal{I}(\nu) = \infty$. See for instance \cite{Comets1989} for a proof.
 
 Since $\mathcal{I}$ has a unique zero at $\mu$, we may conclude that \eqref{eq: to show u n c epsilon} holds. It remains to prove \eqref{eq: rate of convergence of G}. Define $p^{jk}_t \in \mathcal{Y}$ to be such that $p^{jk}_0( J^{jk}_0) = 1$ and $p^{jk}_0(a) = 0$ for $a \neq J^{jk}_0$, and for all $a \in \Gamma_E$,
 \begin{align}
p^{jk}_t(a) =p^{jk}_0(a)+ \sum_{b \in \Gamma_E : b \neq a} \int_0^t \bigg\lbrace p^{jk}_t(b) l_{b\mapsto a}(\tilde{\sigma}^j_s , \tilde{\sigma}^k_s)   - p^{jk}_t(a)    l_{a\mapsto b}(\tilde{\sigma}^j_s , \tilde{\sigma}^k_s)  \bigg\rbrace ds .
\end{align}
 Notice also that, since the evolution of $\lbrace \tilde{\sigma}^l \rbrace_{l\in I_n}$ is independent of the evolution of  $\lbrace \tilde{J}^{jk}_t \rbrace_{j,k\in I_n}$, for $a\in \Gamma_E$,
 \begin{align}
 p^{jk}_t(a) = \mathbb{P}\big( \tilde{J}^{jk}_t = a \; \big| \; \lbrace \tilde{\sigma}^l \rbrace_{l\in I_n} \big).
 \end{align}
 Substituting definitions, we find that
 \begin{align}
 \tilde{G}^{j}_{t}(\alpha,a)- \hat{G}^j_{t}(\alpha,a) &= n^{-1}\sum_{k\in I_n} \mathbf{1}\lbrace \tilde{\sigma}^k_t = \alpha \rbrace  \hat{J}^{jk}_t(a)  \text{ where }\\
 \hat{J}^{jk}_t(a) &= \mathbf{1}\lbrace \tilde{J}^{jk}_t = a \rbrace -  \mathbb{P}\big( \tilde{J}^{jk}_t = a \; \big| \; \lbrace \tilde{\sigma}^l \rbrace_{l\in I_n} \big).
 \end{align}
  Now
 \begin{multline}
 \mathbb{P}\bigg(  \sup_{j\in I_n} \sup_{t\leq T} \sup_{a\in \Gamma_E , \alpha\in\Gamma} \big| \tilde{G}^{j}_{t}(\alpha,a)- \hat{G}^j_{t}(\alpha,a) \big| > \epsilon \bigg) 
 \leq n^2 \sup_{j,p\in I_n}\bigg\lbrace \mathbb{P}\bigg(   \mathcal{V}^{n,j}_{\epsilon}\big( T(p-1) / n\big) \bigg)  \nonumber\\ +|\Gamma| |\Gamma_E|  \sup_{a\in \Gamma_E , \alpha\in\Gamma} \mathbb{P}\bigg( \sup_{ (p-1)T/n \leq t \leq pT/n} \big| \tilde{G}^{j}_{t}(\alpha,a)- \hat{G}^j_{t}(\alpha,a) \big| \geq \epsilon/2  \bigg) \bigg\rbrace
 \end{multline}
and we have defined the event, for $j \in I_n$,
 \begin{align}
 \mathcal{V}^{n,j}_{\epsilon}(t) = \bigg\lbrace \sup_{\alpha\in\Gamma , a\in\Gamma_E} \big| \tilde{G}^{j}_{t}(\alpha,a)- \hat{G}^j_{t}(\alpha,a) \big| \geq \epsilon / 2 \bigg\rbrace .
 \end{align}
 Employing a Chernoff Bound, for a constant $\beta > $,
 \begin{align}
 \mathbb{P}\big( \mathcal{V}^{n,j}_{\epsilon}(t) \big) \leq &\sum_{\alpha\in\Gamma , a\in \Gamma_E}\mathbb{E}\bigg[ \exp\bigg( \beta \sum_{k\in I_n} \mathbf{1}\lbrace \tilde{\sigma}^k_t = \alpha \rbrace  \hat{J}^{jk}_t(a) \bigg) \nonumber \\ &+  \exp\bigg( -\beta \sum_{k\in I_n} \mathbf{1}\lbrace \tilde{\sigma}^k_t = \alpha \rbrace  \hat{J}^{jk}_t(a) \bigg)\bigg] \exp\big( - n \beta \epsilon / 2 \big) \nonumber \\
 \leq &2|\Gamma| |\Gamma_E| \exp\big( C_T n \beta^2 - n \beta \epsilon / 2  \big)
 \end{align}
 for some constant $C_T$ that is independent of $n$ and $\beta$ (as long as $\beta$ is sufficiently small). Substituting $\beta = \epsilon / (4C_T)$, we obtain that
  \begin{align}
 \mathbb{P}\big( \mathcal{V}^{n,j}_{\epsilon}(t) \big) \leq 2|\Gamma| |\Gamma_E| \exp\big( -n \epsilon^2  / (8C_T) \big)
 \end{align}
 Observe that
 \begin{multline}
  \sup_{ (p-1)T/n \leq t \leq pT/n} \big| \tilde{G}^{j}_{t}(\alpha,a)- \hat{G}^j_{t}(\alpha,a) \big| \leq \\ n^{-1} \sum_{k\in I_n} \mathbf{1}\bigg\lbrace \tilde{J}^{jk}_t \neq \tilde{J}^{jk}_{(p-1)T/n} \text{ for some }t \in \big[(p-1)T/n , pT/n \big] \bigg\rbrace
 \end{multline}
For a constant $\beta > 0$,
 \begin{multline*}
 \mathbb{P}\bigg( \sup_{ (p-1)T/n \leq t \leq pT/n} \big| \tilde{G}^{j}_{t}(\alpha,a)- \hat{G}^j_{t}(\alpha,a) \big| \geq \epsilon/2  \bigg) \leq \\ \mathbb{E}\bigg[ \exp\bigg( \beta \sum_{k\in I_n} \mathbf{1}\bigg\lbrace \tilde{J}^{jk}_t \neq \tilde{J}^{jk}_{(p-1)T/n} \text{ for some }t \in \big[(p-1)T/n , pT/n \big] \bigg\rbrace \bigg)\bigg] \exp\big( -\beta \epsilon n / 2 \big)
 \end{multline*}
  Now the probability that $\tilde{J}^{jk}_t \neq \tilde{J}^{jk}_{(p-1)T/n}$ for some $t \in [(p-1)T/n , pT/n]$ scales as $cT/n$ for a universal constant $c$. We thus find that
 \begin{align*}
 \mathbb{P}\bigg( \sup_{ (p-1)T/n \leq t \leq pT/n} \big| \tilde{G}^{j}_{t}(\alpha,a)- \hat{G}^j_{t}(\alpha,a) \big| \geq \epsilon/2  \bigg) &\leq 
 \bigg( 1 + cT/n \big( \exp(\beta) - 1 \big) \bigg)^n \exp\big( - \beta \epsilon n / 2 \big) \\
 &\leq \exp\bigg(   cT \big( \exp(\beta)-1 \big) - \beta \epsilon n / 2 \bigg).
 \end{align*}
 Choosing $\beta = \log n$, we find that for large enough $n$,
  \begin{align*}
\sup_{p,j\in I_n} \mathbb{P}\bigg( \sup_{ (p-1)T/n \leq t \leq pT/n} \big| \tilde{G}^{j}_{t}(\alpha,a)- \hat{G}^j_{t}(\alpha,a) \big| \geq \epsilon/2  \bigg) \leq \exp\big( - n \big)
 \end{align*}
 Combining the above estimates, we find that
 \begin{align*}
\lsup{n} n^{-1}\log  \mathbb{P}\bigg(  \sup_{j\in I_n} \sup_{t\leq T} \sup_{a\in \Gamma_E , \alpha\in\Gamma} \big| \tilde{G}^{j}_{t}(\alpha,a)- \hat{G}^j_{t}(\alpha,a) \big| > \epsilon \bigg) < 0,
 \end{align*}
 as required.
  \end{proof}

 \subsection{Coupling of Systems}
 
 Next, we outline a coupling of the $n$-dimensional systems in the same space (this roughly means that the probability laws of $\lbrace \sigma^j_t \rbrace_{j\in I_n}$ and $\lbrace \tilde{\sigma}^j_t \rbrace_{j\in I_n}$ remain the same, but the system is constructed to maximize the correlations as much as possible). To do this, we will employ the time-rescaling formalism that has been extensively employed by Kurtz \cite{Ethier1986} and Anderson \cite{Anderson2015} (amongst others).  
 
To this end, let   
  \begin{align}
    \big\lbrace \grave{Y}_{\alpha\mapsto \beta}^j(\cdot) , \bar{Y}_{\alpha\mapsto \beta}^j(\cdot)  , \breve{Y}^j_{\alpha\mapsto\beta}(\cdot) \big\rbrace_{\alpha,\beta\in\Gamma \fatsemi \alpha \neq \beta \fatsemi j\in I_n } \subset \mathcal{D}\big( [0,\infty) , \mathbb{Z}^+ \big) \label{eq: counting one}
 \end{align}
be independent unit intensity Poisson Processes. For the edge dynamics, we analogously let
 \begin{align}
  \big\lbrace \grave{Y}_{a \mapsto b}^{jk}(\cdot ) , \bar{Y}_{a \mapsto b}^{jk}(\cdot)  , \breve{Y}^{jk}_{a\mapsto b}(\cdot)\big\rbrace_{a,b \in \Gamma_E \fatsemi a \neq b \fatsemi j,k\in I_n} \label{eq: counting two}
 \end{align}
  be another set of mutually independent unit intensity Poisson Processes. We are going to define the stochastic processes $\lbrace \sigma^j_t , \tilde{\sigma}^j_t , \hat{G}^j_{\zeta,t} \rbrace_{j\in I_n , \zeta \in \mathcal{E}}$ to be functions of the above processes. To this end, we first note that the initial conditions are as previously specified, i.e. $\tilde{\sigma}^j_0 = \sigma^j_0$, $\tilde{J}^{jk}_0 = J^{jk}_0$. As previously, we define 
 \begin{equation}
\hat{G}^j_{t} = \psi_{\theta^j_n,t}( \tilde{\sigma}^j ,\mu_t)
 \end{equation}
We next define $\big\lbrace Z^j_{\alpha\mapsto \beta}(t) , \tilde{Z}^j_{\alpha\mapsto \beta}(t) , \grave{Z}^{j}_{\alpha\mapsto\beta} ( t ) , \breve{Z}^{j}_{\alpha\mapsto\beta} ( t ) \big\rbrace_{j\in I_n}$ to be counting processes, and such that 
 \begin{align}
Z^j_{\alpha\mapsto \beta}(t) =& \grave{Z}^{j}_{\alpha\mapsto\beta} ( t ) + \bar{Y}^{j}_{\alpha\mapsto\beta} \bigg( \int_0^t \bar{f}^j_{\alpha\mapsto\beta}(s) ds  \bigg) \\
\tilde{Z}^j_{\alpha\mapsto \beta}(t) =&  \breve{Z}^{j}_{\alpha\mapsto\beta} (t)+ \bar{Y}^{j}_{\alpha\mapsto\beta} \bigg( \int_0^t \bar{f}^j_{\alpha\mapsto\beta}(s) ds  \bigg) \\
\grave{Z}^{j}_{\alpha\mapsto\beta} ( t ) =& \grave{Y}^{j}_{\alpha\mapsto\beta} \bigg( \int_0^t \grave{f}^j_{\alpha\mapsto\beta}(s) ds  \bigg) \\
\breve{Z}^{j}_{\alpha\mapsto\beta} ( t ) =& \breve{Y}^{j}_{\alpha\mapsto\beta} \bigg( \int_0^t \breve{f}^j_{\alpha\mapsto\beta}(s) ds  \bigg) \\
 \bar{f}^j_{\alpha \mapsto \beta}(t) =& \inf\big\lbrace f_{\alpha\mapsto\beta}(G^j_t) \mathbf{1}\lbrace \sigma^j_t = \alpha \rbrace  , f_{\alpha\mapsto\beta}(\hat{G}^j_t  ) \mathbf{1}\lbrace \tilde{\sigma}^j_t = \alpha \rbrace  \big\rbrace \\
\grave{f}^j_{\alpha\mapsto \beta}(t) =&  f_{\alpha\mapsto\beta}(G^j_t) \mathbf{1}\lbrace \sigma^j_t = \alpha \rbrace   -  \bar{f}^j_{\alpha \mapsto \beta}(t) \\
\breve{f}^j_{\alpha\mapsto \beta}(t) =&  f_{\alpha\mapsto\beta}(\hat{G}^j_t ) \mathbf{1}\lbrace \tilde{\sigma}^j_t = \alpha \rbrace  -  \bar{f}^j_{\alpha \mapsto \beta}(t) .
 \end{align}
The edge variables are coupled in an analogous manner to the node variables.
  \begin{align}
Z^{jk}_{a\mapsto b}(t) =& \grave{Z}^{jk}_{a\mapsto b}(t) + \bar{Y}^{jk}_{a\mapsto b} \bigg( \int_0^t \bar{l}^j_{\alpha\mapsto\beta}(s) ds  \bigg) \\
\tilde{Z}^{jk}_{a\mapsto b}(t) =& \breve{Z}^{jk}_{a\mapsto b}(t)+ \bar{Y}^{jk}_{a\mapsto b} \bigg( \int_0^t \bar{l}^j_{\alpha\mapsto\beta}(s) ds  \bigg) \\
\grave{Z}^{jk}_{a\mapsto b}(t) =&  \grave{Y}^{jk}_{a\mapsto b} \bigg( \int_0^t \grave{l}^{jk}_{a\mapsto b}(s) ds  \bigg) \\
\breve{Z}^{jk}_{a\mapsto b}(t) =&  \breve{Y}^{jk}_{a\mapsto b} \bigg( \int_0^t \breve{l}^{jk}_{a\mapsto b}(s) ds  \bigg) \\
 \bar{l}^{jk}_{a \mapsto b }(t) =& \inf\big\lbrace l_{a \mapsto b}(\sigma^j_t , \sigma^k_t) \mathbf{1}\lbrace J^{jk}_t = a \rbrace , l_{a\mapsto b}(\tilde{\sigma}^j_t , \tilde{\sigma}^k_t) \mathbf{1}\lbrace \tilde{J}^{jk}_t = a \rbrace  \big\rbrace \\
\grave{l}^{jk}_{a \mapsto b}(t) =& l_{a \mapsto b}(\sigma^j_t , \sigma^k_t) \mathbf{1}\lbrace J^{jk}_t = a \rbrace   -  \bar{l}^{jk}_{a \mapsto b }(t)  \\
\breve{l}^{jk}_{a \mapsto b}(t) =& l_{a \mapsto b}(\tilde{\sigma}^j_t , \tilde{\sigma}^k_t) \mathbf{1}\lbrace \tilde{J}^{jk}_t = a \rbrace  -   \bar{l}^{jk}_{a \mapsto b }(t) .
 \end{align}
 We then stipulate that, for any $\alpha \in \Gamma$,
 \begin{align}
  \sigma^j_t &=  \alpha \text{ if and only if } \\
 \gamma^j_t(\alpha) &:=  \mathbf{1}\lbrace \sigma^j_0 = \alpha \rbrace + \sum_{\beta \neq \alpha} \big( Z^j_{\beta \mapsto \alpha}(t) - Z^j_{\alpha\mapsto \beta}(t) \big) = 1.
  \end{align}
  Note that only one of $\lbrace \gamma^j_t(\alpha) \rbrace_{\alpha\in \Gamma}$ can equal one, and the rest must be zero. This is because $\gamma^j_t(\alpha)$ counts the net number of transitions to the state $\alpha$ minus the number of transitions away, plus $1$ if $\sigma^j_0$ is $\alpha$. We similarly stipulate that
    \begin{align}
  \tilde{\sigma}^j_t &=  \alpha \text{ if and only if }\\
   \mathbf{1}\lbrace \sigma^j_0 = \alpha \rbrace + \sum_{\beta \neq \alpha} \big( \tilde{Z}^j_{\beta \mapsto \alpha}(t) - \tilde{Z}^j_{\alpha\mapsto \beta}(t) \big) &= 1.
  \end{align}
For the edge variables, we stipulate that, for any $a\in \Gamma_E$
  \begin{align}
 J^{jk}_{a\mapsto b}(t) &= a \text{ in and only if  } \\
  \mathbf{1}\lbrace J^{jk}_0  = a \rbrace + \sum_{b \neq a} \big( Z^{jk}_{b \mapsto a}(t) - Z^{jk}_{a\mapsto b}(t) \big) &= 1
 \end{align}
and we stipulate that 
 \begin{align}
 \tilde{J}^{jk}_{a\mapsto b}(t) &= a \text{ if and only if } \\
 \mathbf{1}\lbrace J^{jk}_0 = a \rbrace + \sum_{b \neq a} \big( \tilde{Z}^{jk}_{b \mapsto a}(t) - \tilde{Z}^{jk}_{a\mapsto b}(t) \big) &= 1 .
 \end{align}

 \begin{lemma}
 The above system is well-defined and consistent with the earlier definitions.
 \end{lemma}
 \begin{proof}
 The fact that the variables $\big\lbrace Z^j_{\alpha\mapsto \beta}(t) , \tilde{Z}^j_{\alpha\mapsto \beta}(t) , \grave{Z}^{j}_{\alpha\mapsto\beta} ( t ) , \breve{Z}^{j}_{\alpha\mapsto\beta} ( t ) , \sigma^j_t , \tilde{\sigma}^j_t \big\rbrace_{j\in I_n}$ are well-defined follows from the fact that the counting processes \eqref{eq: counting one} - \eqref{eq: counting two} are piecewise-constant, with only a finite number of jumps over a finite time interval. For the time-intervals between jumps, there is thus existence and uniqueness due to the Picard-Lindelof Theorem for ODEs. We note also that these stochastic variables are adapted to the same filtration.
 \end{proof}

 We need to control the distance between the two systems. To this end, for $t\leq T$, define 
 \begin{align}
 \delta^n_t =& n^{-1} z^n_t \text{ where } \label{eq: delta n t definition} \\
z^n_t =&  \sum_{j\in I_n} \sum_{\alpha,\beta \in \Gamma : \alpha\neq \beta} \big( \grave{Z}^j_{\alpha\mapsto\beta}(t) + \breve{Z}^j_{\alpha\mapsto\beta}(t) \big) \\
 u^{n,j}_t =& \sum_{k\in I_n} \sum_{a,b \in \Gamma_E: a\neq b} \big( \grave{Z}^{jk}_{a \mapsto b}(t) + \breve{Z}^{jk}_{a \mapsto b}(t) \big)  \\
  \varphi^{n}_t =& n^{-2} u^n_t \text{ where } \\
 u^{n}_t =& \sum_{j,k\in I_n} \sum_{a,b \in \Gamma_E: a\neq b} \big( \grave{Z}^{jk}_{a \mapsto b}(t) + \breve{Z}^{jk}_{a \mapsto b}(t) \big)  \\
 \eta^n_t =& n^{-1} \sum_{j\in I_n}  \| \tilde{G}^j_{t}- \hat{G}^j_{t} \|\label{eq: eta n t definition} 
 \end{align} 
The main result that we will prove in this section is the following Lemma.
 \begin{lemma} \label{Lemma to prove coupling lemma}
 For any $\epsilon > 0$, there exists $\tilde{\epsilon} > 0$ such that
 \begin{align}
 \lsup{n} n^{-1} \log \mathbb{P}\big( \delta^n_T \geq \epsilon \; \;  , \; \; \sup_{t\leq T}  \eta^n_t \leq \tilde{\epsilon} \big) < 0.
 \end{align}
 \end{lemma}
Before we prove Lemma \ref{Lemma to prove coupling lemma}, let us first see why it implies the veracity of Theorem \ref{Main Theorem}.
 \begin{proof}
 Thanks to the triangle inequality
\[
 d_W\big( \hat{\mu}^n_{T} , \mu_T \big)  \leq  d_W\big( \tilde{\mu}^n_{T} , \mu_T \big) + d_W\big( \tilde{\mu}^n_{T} , \hat{\mu}^n_T \big) .
\]
Hence
  \begin{multline}
 \lsup{n} n^{-1} \log \mathbb{P}\big(  d_W\big( \hat{\mu}^n_{T} , \mu_T \big)  \geq \epsilon  ,  \sup_{t\leq T}  \eta^n_t \leq \tilde{\epsilon}  \big) \leq\\  \max\bigg\lbrace   \lsup{n} n^{-1} \log \mathbb{P}\bigg(  d_W\big( \tilde{\mu}^n_{T} , \mu_T \big)  \geq \epsilon /2 ,  \sup_{t\leq T}  \eta^n_t \leq \tilde{\epsilon}  \bigg) , \\  \lsup{n} n^{-1} \log \mathbb{P}\bigg(  d_W\big( \hat{\mu}^n_{T} , \tilde{\mu}^n_T \big)  \geq \epsilon /2 \; \; , \; \;  \sup_{t\leq T}  \eta^n_t \leq \tilde{\epsilon} \bigg) \bigg\rbrace \label{eq: multline split the probabilities}
 \end{multline}
Thanks to Lemma \ref{Lemma convergence of tilde mu n t}, the first term on the RHS is strictly negative. It also implies that $\mathbb{P}$-almost-surely,
\[
\lsup{n} \eta^n_t \leq \tilde{\epsilon}.
\]

Lemma \ref{Lemma to prove coupling lemma} implies that the second term on the RHS of \eqref{eq: multline split the probabilities} is negative. We therefore find that 
\[
 \lsup{n} n^{-1} \log \mathbb{P}\big(  d_W\big( \hat{\mu}^n_{T} , \mu_T \big)  \geq \epsilon  ,  \sup_{t\leq T}  \eta^n_t \leq \tilde{\epsilon}  \big) < 0.
\]
 The Borel-Cantelli Lemma now implies that, $\mathbb{P}$-almost-surely, for all large enough $n$,
 \[
 d_W\big( \hat{\mu}^n_{T} , \mu_T \big) < \epsilon
 \]
 \end{proof}
 The rest of this paper is dedicated to proving Lemma \ref{Lemma to prove coupling lemma}. To this end, we must first prove that the intensities of the `remainder' Poisson Processes can be uniformly bounded in terms of $\delta^n_t$ and $\varphi^{n}_t$. 
\begin{lemma} \label{Lemma bound the intensities}
There exists a constant $c_T$ (the constant is independent of $n$ and $j$) such that for all $t\leq T$, all $j\in I_n$,
\begin{align}
n^{-1} \sum_{k\in I_n} \sum_{\alpha,\beta \in \Gamma \; : \; \alpha\neq \beta} \big( \grave{f}^k_{\alpha\mapsto\beta}(t) + \breve{f}^k_{\alpha\mapsto\beta}(t) \big) &\leq c_T \big(  \delta^n_t +  \varphi^{n}_t + \eta^n_t \big) \label{Lemma first intensity bound} \\
n^{-2} \sum_{j,k\in I_n} \sum_{a,b \in \Gamma_E \; : \; a\neq b } \big( \grave{l}^{jk}_{a \mapsto b}(t) + \breve{l}^{jk}_{a \mapsto b}(t)  \big) &\leq c_T \big( \delta^n_t + \varphi^{n}_t \big)\label{Lemma second intensity bound} \\
n^{-1} \sum_{j\in I_n} \big\| G^j_t - \hat{G}^j_t \big\|  \leq & c_T\eta^n_t + c_T \big( \delta^n_t + \varphi^{n}_t \big).
\end{align}
\end{lemma}
\begin{proof}
Since the function $f_{\alpha\mapsto \beta}$ is Lipschitz and bounded, there is a constant $c > 0$ such that
\begin{align}
n^{-1} \sum_{k\in I_n} \sum_{\alpha,\beta \in \Gamma \; : \; \alpha\neq \beta} \big( \grave{f}^k_{\alpha\mapsto\beta}(t) + \breve{f}^k_{\alpha\mapsto\beta}(t) \big)
\leq n^{-1} c\sum_{k\in I_n}\big( \| G^k_t - \hat{G}^k_t \| + \mathbf{1}\lbrace \sigma^k_t \neq \tilde{\sigma}^k_t \rbrace \big)
\end{align}
Furthermore
\[
n^{-1} \sum_{k\in I_n}  \mathbf{1}\lbrace \sigma^k_t \neq \tilde{\sigma}^k_t \rbrace  \leq | \Gamma | \delta^n_t
\]
Furthermore, thanks to the triangle inequality,
\begin{align}
\| G^k_t - \hat{G}^k_t \| \leq & \| \tilde{G}^k_t - \hat{G}^k_t \| + \| G^k_t - \tilde{G}^k_t \|\nonumber \\
\leq & \| \tilde{G}^k_t - \hat{G}^k_t \| + \rm{Const}\times n^{-1}\sum_{q\in I_n}  \mathbf{1}\lbrace \sigma^q_t \neq \tilde{\sigma}^q_t \rbrace + \rm{Const}  \times n^{-1} u^{n,k}_t \nonumber \\
\leq & \| \tilde{G}^k_t - \hat{G}^k_t \| + \rm{Const} \times n^{-1} u^{n,k}_t + \rm{Const} \times \delta^n_t .
\end{align}
Combining these bounds, we obtain \eqref{Lemma first intensity bound}. The proof of \eqref{Lemma second intensity bound}  is analogous.
\end{proof}

For a large positive integer $m$, let $t^{(m)}_i = i T / m$ and for any $t \in [0,T]$, $t^{(m)} := \sup\big\lbrace s \leq t \; : s = t^{(m)}_i \text{ for some }i \leq m \big\rbrace$. 
Write
\begin{align}
p_n(t) =  \max\lbrace  n^{-1}\tilde{z}_n(t) , n^{-2} \tilde{u}_n(t) \rbrace
\end{align}
For a constant $a > 0$, let $\tilde{\epsilon}_m > 0$ be such that (i)
 \[
\tilde{\epsilon}_m >  c_T \big( T \tilde{\epsilon} / m + T \epsilon \exp\big(a T \big) / m \big),
\]
and (ii) $\lim_{m\to\infty} \tilde{\epsilon}_m = 0$. Define the stopping times
\begin{align}
\tau_m =& \inf\big\lbrace t \leq T \; : \;  p_n(t) \geq \tilde{\epsilon}_m + \epsilon \exp( a t) \big\rbrace
\end{align}
Observe that
\begin{multline}
\bigg\lbrace p_n( t^{(m)}_b)\leq \epsilon \exp\big(a t^{(m)}_b\big) \bigg\rbrace \bigcap \bigg\lbrace   \sup_{t\in [t^{(m)}_b \wedge \tau_m , t^{(m)}_{b+1} \wedge \tau_m ]} \big\lbrace \tilde{z}_n\big(t \big) - \tilde{z}_n\big( t^{(m)}_b \big) \big\rbrace \leq n\tilde{\epsilon}_m \bigg\rbrace\\
\subseteq \bigg\lbrace   p_n( t^{(m)}_{b+1})\leq \epsilon \exp\big(a t^{(m)}_{b+1}\big)  \bigg\rbrace. \label{eq: decomposition of events}
\end{multline}
\begin{lemma} \label{Lemma exponential decay of increment}
For all $m \geq 2$ and $0\leq b \leq m-1$,
\begin{multline}
\lsup{n} n^{-1}\log \mathbb{P}\bigg(  \sup_{t\in [t^{(m)}_b \wedge \tau_m , t^{(m)}_{b+1} \wedge \tau_m ]} \big\lbrace \tilde{z}_n\big(t \big) - \tilde{z}_n\big( t^{(m)}_b \big) \big\rbrace \geq n\tilde{\epsilon}_m \text{ or } \\
  \sup_{t\in [t^{(m)}_b  \wedge \tau_m , t^{(m)}_{b+1} \wedge \tau_m]} \big\lbrace \tilde{u}_n\big(t \big) - \tilde{u}_n\big( t^{(m)}_b \big) \big\rbrace \geq n^2 \tilde{\epsilon}_m \;  \;
 \bigg) < 0.
\end{multline}
\end{lemma}
\begin{proof}
For all times $t\leq  \tau \wedge \tau_m $, $ \tilde{z}_n\big(t \big) - \tilde{z}_n\big( t^{(m)}_b \big) $ is Poissonian, with intensity upperbounded by
\[
n c_T  \big( \tilde{\epsilon}_m  +   \epsilon   \exp(a T )  \big).
\]
Similarly, $ \tilde{u}_n\big(t \big) - \tilde{u}_n\big( t^{(m)}_b \big) $ is Poissonian, with intensity upperbounded by
\[
n^2 c_T \big( \tilde{\epsilon}_m  +   \epsilon  \exp(a T )  \big).
\]
Since $\tilde{\epsilon}_m > \frac{T}{m} c_T \big(\tilde{\epsilon}_m  +   \epsilon \exp(a T )  \big) $, the result is now a standard result from Poisson Processes \cite{Ethier1986}.
\end{proof}
We obtain an immediate corollary.
\begin{corollary} \label{Corollary small tau m}
\[
\lsup{n} n^{-1}\log \mathbb{P}\big( \tau_m < T \big) < 0.
\]
\end{corollary}
\begin{proof}
Thanks to \eqref{eq: decomposition of events},
\begin{multline}
\lsup{n} n^{-1}\log \mathbb{P}\big( \tau_m < T \big) \leq \\ \sup_{0\leq b \leq m-1}\lsup{n} n^{-1}\log \mathbb{P}\bigg( \sup_{t\in [t^{(m)}_b \wedge \tilde{\tau}_m , t^{(m)}_{b+1} \wedge \tilde{\tau}_m]} \big\lbrace \tilde{z}_n\big(t \big) - \tilde{z}_n\big( t^{(m)}_b \big) \big\rbrace \geq n\tilde{\epsilon}_m  \; \; , \; \; t^{(m)}_b \leq \tau_m  \bigg) \\
< 0,
\end{multline}
thanks to Lemma \ref{Lemma exponential decay of increment}.
\end{proof}
We can now prove Lemma \ref{Lemma to prove coupling lemma}.
\begin{proof}
For any $\kappa > 0$, we can take $\epsilon$ small enough, and $m$ large enough, that
\[ 
\tilde{\epsilon}_m + \epsilon \exp( a T) < \kappa.
\]
It then follows from Corollary \ref{Corollary small tau m} that
\[
\lsup{n} n^{-1}\log \mathbb{P}\big( p_n(T) \geq \kappa \big) < 0.
\]
This establishes the Lemma.
\end{proof}

\bibliographystyle{plain}
\bibliography{bib2,adaptiveadditional}

\end{document}